\documentclass[a4paper]{amsart}
\usepackage{amssymb,amsfonts,amsthm,amsmath}
\usepackage{euscript}

\usepackage{amsmath}
\usepackage{amsbsy}
\usepackage{amssymb}
\usepackage{amscd}
\usepackage{a4wide}
\usepackage{epic}

\newtheorem{theorem}{Theorem}
\newtheorem{proposition}{Proposition}[section]

\theoremstyle{definition}

\theoremstyle{remark}
\newtheorem*{remark}{Remark}

\DeclareMathOperator{\Imp}{\mathrm{Im}\,}

\DeclareMathOperator{\End}{\mathrm{End}}
\DeclareMathOperator{\Hom}{\mathrm{Hom}}
\DeclareMathOperator{\Idem}{\mathrm{Idem}}

\DeclareMathOperator{\deff}{\mathrm{def}}

\def\Id{\mathrm{Id}\,}

\title{$n$-Subspaces in linear and unitary spaces}
\author{Yu. S. Samoilenko} 

\address{Institute of
  Mathematics, National Academy of Sciences of Ukraine, 3
  Tereshchenkivs'ka, Kyiv, 01601, Ukraine}
\email{yurii\_sam@imath.kiev.ua}

\author{D. Y. Yakymenko}
\address{Institute of Mathematics, National Academy of
  Sciences of Ukraine, 3 Tereshchenkivs'ka, Kyiv, 01601,
  Ukraine}
\email{dandan.ua@gmail.com}

\date{}

\begin{document}

\newcommand{\iiSC}[1]{\stackrel{\circ}{#1}}
\newcommand{\iiSB}[1]{\stackrel{\bullet}{#1}}

\maketitle

\begin{abstract}
  We study a relation between brick $n$-tuples of subspaces of a
  finite dimensional linear space, and irreducible $n$-tuples of
  subspaces of a finite dimensional Hilbert (unitary) space such
  that a linear combination, with positive coefficients, of
  orthogonal projections onto these subspaces equals the identity
  operator. We prove that brick systems of one-dimensional subspaces
  and the systems obtained from them by applying the Coxeter
  functors (in particular, all brick triples and quadruples of
  subspaces) can be unitarized. For each brick triple and
  quadruple of subspaces, we describe sets of characters that admit
  a unitarization.
\end{abstract}

\section{Introduction}

A relationship between representations of groups on linear spaces
and their unitary representations on Hilbert spaces is useful for
the two kinds of representations.

In this paper, we study a relation between brick $n$-tuples
$L=(V;V_1,\dots ,V_n)$ of subspaces $V_k$ of a complex finite
dimensional linear space $V$, see Section~2, and irreducible
orthoscalar $n$-tuples $S=(H;H_1, \dots H_n)$ of subspaces $H_k$ of
a finite dimensional Hilbert (unitary) space $H$, that is, such that
there exists a collection of positive numbers $(a_0;a_1,\dots
,a_n)$, called a character, such that
\begin{equation}
  \label{eq:1}
  \sum a_k P_{H_k} = a_0 I, 
\end{equation}
where $P_{H_k}$ are orthogonal projections onto the subspaces $H_k$
and $I$ is the identity operator on $H$, see Section~4. Recall that
an $n$-tuple $L=(V;V_1,\dots ,V_n)$ of subspaces $V_k$ is called brick
if any linear operator $X:V\rightarrow V$ such that $X(V_k) \subset
V_k$ is a multiple of the identity operator. An $n$-tuple of
orthogonal projections $\{P_{H_k}\}_{k=1}^n$ is called irreducible
if, for any linear operator $X:H\rightarrow H$,
$[X,P_{H_k}]=0,~k=1,\dots ,n$, implies that $X=\lambda I$, $\lambda \in
\mathbb{C}$.

If an $n$-tuple of orthogonal projections $P_{H_k}$ on $H$ is
irreducible and satisfies relation~(\ref{eq:1}), then the
corresponding collection of the subspaces $H_k$ of the linear space
$H$ will be brick, see~\cite{KNR}. In this paper, we call a brick
collection $L=(V;V_1,\dots ,V_n)$ unitarizable if there exists a
scalar product on $V$ and a character $\chi=(a_0; a_1,\dots ,a_n)$ such
that the corresponding collection of orthogonal projections onto
$H_k=V_k$ satisfies~(\ref{eq:1}). In Section~4 we prove
(Theorems~\ref{Th:1.1} and~\ref{Th:1.2}) that brick systems of
one-dimensional subspaces and the ones they yield by applying the
Coxeter functors (in particular, all brick triples and quadruples of
spaces) can be unitarized, see~\cite{MS1}, \cite{MS2} for
unitarizing all nondegenerate brick quadruples with the
characters $(\gamma; 1,1,1,1), ~\gamma>0$.  There are also other
examples of systems of subspaces that can be unitarized. In
Section~5, for all brick quadruples of subspaces we describe sets of
characters that allow for a unitarization, see Theorems~\ref{Th:2.1}
and~\ref{Th:2.2}.

The interests of the authors to the topics discussed in the paper
was increased in connection with the article~\cite{EW}, where it was
remarked that ``There seems to be interesting relations with the
n-tuples of subspaces and the sums of projections''.

\section{On $n$-tuples of subspaces of a linear space}

\subsection{}
In this subsection, we recall known facts about $n$-tuples of
subspaces of a linear space, needed in the sequel.

Let $L = (V; V_1,V_2, \dots ,V_n)$ be a system of subspaces of $V$,
$\tilde{L} = (\tilde{V}; \tilde{V_1},\tilde{V_2}, \dots
,\tilde{V_n})$ a system of subspaces of $\tilde{V}$. A linear
operator $R:V\rightarrow\tilde{V}$ is called a homomorphism of the
system $L$ into $\tilde{L}$ if $R(V_i) \subset \tilde{V_i}$,
$\forall i=\overline{1,n}$. $R:V\rightarrow\tilde{V}$ is called an
isomorphism if there exists an inverse $R^{-1}$ such that
$R^{-1}(\tilde{V_i}) \subset V_i$, $\forall i=\overline{1,n}$, and
the systems $L$ and $\tilde{L}$ will be called isomorphic
(equivalent).

Denote by $\Hom(L, \tilde{L})$ the set of homomorphisms from $L$
into $\tilde{L}$. $\End(L):=\Hom(L,L)$, that is, $\End(L) =
\{R:V\rightarrow V\mid R(V_i) \subset V_i, \forall i=\overline{1,n}
\}$. A system $S$ is called brick (Schur, transitive) if
$\End(L)=\mathbb C I$.

Denote by $\Idem(L) = \{R:V\rightarrow V\mid R(V_i) \subset V_i,
\forall i=\overline{1,n}, R^{2} = R \}$. A system $L$ is called
indecomposable if $\Idem(L)=\{0, I\}$. The property of being
indecomposable is equivalent to that the system is not isomorphic to
a direct sum of two nonzero systems.

It directly follows from the definitions that a brick system is also
indecomposable. However, if $n \geqslant 4$ there are examples
showing that the converse is not true.

An isomorphism preserves the property of a system to be
indecomposable or brick.

\subsection{}
There are only four indecomposable nonequivalent pairs of subspaces,
--- $(\mathbb C; 0, 0)$, $(\mathbb C; \mathbb C, 0)$, $(\mathbb C;
0, \mathbb C)$, $(\mathbb C; \mathbb C, \mathbb C)$. All of them are
brick.

\sloppy
The number of nonequivalent indecomposable triples of subspaces is
$9$. There are eight triples of subspaces of a one-dimensional
space, --- $(\mathbb C; 0, 0, 0)$, $(\mathbb C; \mathbb C, 0, 0)$,
$(\mathbb C; 0, \mathbb C, 0)$, $(\mathbb C; \mathbb C, \mathbb C,
0)$, $(\mathbb C; 0, 0, \mathbb C)$, $(\mathbb C; \mathbb C, 0,
\mathbb C)$, $(\mathbb C; 0, \mathbb C, \mathbb C)$, $(\mathbb C;
\mathbb C, \mathbb C, \mathbb C)$, and one triple in a
two-dimensional space. This is $(\mathbb C^2; \mathbb C(1,0), \mathbb
C(0,1), \mathbb C (1,1))$. All of them are brick.

\fussy
For $n=4$ already, not every indecomposable $n$-tuple will be brick.
A description of brick quadruples and indecomposable quadruples is
given in~\cite{Br,Naz,GP} and others. For our
purposes, a complete description is not needed, but we will only use
some properties.

Let $d=(d_0;d_1,d_2,d_3,d_4)$ be a generalized dimension of the
system $L=(V;V_1, V_2, V_3, V_4)$. A Tits form is the quadratic form
$$
T(d)=\sum_{i=0}^4d_i^2-d_0\sum_{i=1}^4d_i.
$$

For an indecomposable system $L$, the Tits form of the dimension $d$
equals either $1$ (the dimension $d$ in such a case is called a real
root) or $0$ (and $d$ is called an imaginary root). If $d$ is a real
root, then for this dimension there exists exactly one
indecomposable quadruple of subspaces, which is a brick quadruple.
If $d$ is an imaginary root, then for this dimension there exists a
family of quadruples. Imaginary roots will be multiples of the
imaginary root $\sigma = (2;1,1,1,1)$. In such a case, brick systems
will be obtained only for the minimal root $\sigma$.

To classify quadruples of spaces, it is convenient to use the notion
of a deficiency defined by 
$$
\deff(L)=2d_0-\sum_{i=1}^4d_i.
$$

One way to construct indecomposable quadruples is to use Coxeter
functors $\Phi ^+$ and $\Phi ^-$~\cite{GP}. These functors allow to
use a system to obtain other systems preserving the
indecomposability and brick properties. They also preserve the
deficiency and the type of the root.

The following is a list of all dimensions corresponding to real
roots:
\begin{itemize}
\item[] $D_4(2m+1, -1) = (2m+1; m,m,m,m+1), ~ \deff=-1$,
\item[] $D_4(2m+1, 1) = (2m+1; m+1,m+1,m+1,m), ~ \deff=1$,
\item[] $D_4(2m, -1) = (2m; m,m,m,m-1), ~ \deff=-1$,
\item[] $D_4(2m, 1) = (2m; m,m,m,m+1), ~ \deff=1$,
\end{itemize}
and permutation of the subspaces $D_i( \cdot, \cdot), i=1,2,3$;
\begin{itemize}
\item[] $D_0(2m+1, -2) = (2m+1; m,m,m,m), ~ \deff=-2$,
\item[] $D_0(2m, 2) = (2m+1; m+1,m+1,m+1,m+1), ~ \deff=2$,
\item[] $D_{3,4}(2m+1, 0) = (2m+1; m,m,m+1,m+1), ~ \deff=0$,
\end{itemize}
and permutations of the subspaces $D_{i,j}(2m+1, 0)$.

In the case where $\deff\neq0$, all indecomposable systems with
these dimensions will be brick and can be obtained by applying the
Coxeter functors to the simplest collections of subspaces; these are
collections of subspaces of a space of dimension $1$.  For the
dimension $D_{i,j}(2m+1, 0)$, a system will be brick only if $m=0$.

Brick nonequivalent quadruples in a space of dimension $\sigma =
(2;1,1,1,1)$ can be written as follows:
$$
\begin{array}{rcl}
  S_{\mu} &=& (\mathbb C^2=<e_1, e_2>; <e_1>, <e_2>, <e_1+\mu
  e_2>, <e_1+e_2>),\qquad \mu \in \mathbb C \setminus \{0,1\},\\
  S_{3,4} &=& (\mathbb C^2=<e_1, e_2>; <e_1>, <e_2>, <e_1+e_2>,
  <e_1+e_2>)
\end{array}
$$
and permutations of $S_{i,j}$.

\subsection{}
For $n\geqslant5$ a description of indecomposable $n$-tuples of
subspaces is a very difficult problem and contains, as a
subproblem, the problem of a description, up to unitary
equivalence, of indecomposable pairs of operators on a finite
dimensional linear space.

Indeed, let $E$ be a linear space, $A$, $B$ linear operators on $E$.
Consider quinaries of subspaces, $L_{(A,B)}=(E\oplus E;
(x,0),(0,x),(x,x),(x,Ax),(x,Bx)$, $x\in E)$. Such a quinary  will
be called an operator quinary. Note that each quinary
$(V;V_1,\dots ,V_5)$ such that $\dim V=2n$, $\dim V_i=n$,
$i=\overline{1,5}$, and $V_i\cap V_j=0$, $i\neq j$, is equivalent to
an operator quinary  with nondegenerate $A$ and $B$.

A description of indecomposable quinaries up to equivalence is
already a very difficult problem, --- it is the problem of a
description of indecomposable pairs of operators on a linear space
up to equivalence.

\begin{proposition}\label{P:1}
  \begin{enumerate}
  \item[1)] $L_{A,B} \simeq L_{\tilde{A},\tilde{B}}$
    $~\Longleftrightarrow~$ $(A,B) \thicksim (\tilde{A},\tilde{B})$,
    that is, there exists an invertible operator $T$ from
    $\tilde{E}$ into $E$ such that $\tilde{A}=T^{-1}AT,
    \tilde{B}=T^{-1}BT$.
  \item[2)] $L_{A,B}$ is indecomposable $~\Longleftrightarrow~$
    the pair $(A,B)$ is indecomposable, that is, for all idempotents
    $T=T^2$ on $E$ such that $TA=AT$, $TB=BT$, we have that $T=0$ or
    $T=I$.
  \end{enumerate}
\end{proposition}

For $n\geqslant5$, the problem of a description of brick $n$-tuples
of subspaces up to equivalence is also very difficult.

\begin{proposition}\label{P:2} 
  $L_{A,B}$ is brick if and only if $(A,B)$ is brick, that is, if
  $TA=AT, ~TB=BT$, then $T=\lambda I$.
\end{proposition}

This problem, for example, contains the problem of describing
irreducible pairs of unitary operators on a finite dimensional
Hilbert space up to unitary equivalence~(\cite{MS2}).

\subsection{} 
A possible additional condition for the problem of describing
indecomposable $n$-tuples of subspaces to become meaningful is the
condition that the subspaces of the collection make a representation
of a finite partially ordered set, see papers on representations
of partially ordered sets in the category of linear
spaces~\cite{NR,Kl} and others.

Another additional condition for the problem of a description of
irreducible $n$-tuples of subspaces to become solvable is a
condition on possible indecomposable terms in the decomposition of the
quadruples $(V;V_{i_1},V_{i_2},V_{i_3},V_{i_4})$,
$i_1,i_2,i_3,i_4\in\{1,2,\dots, n\}$, $i_k\neq i_j$ $(k\neq j)$.

\section{On $n$-tuples of subspaces of a Hilbert 
  space}\label{Sec:3}

\subsection{}
There are many works dealing with $n$-tuples
$S=(H;H_1,H_2,\ldots,H_n)$ of subspaces of a Hilbert space $H$,
see~\cite{Ha1,Sund,EW} and others. In the sequel, $H$ is usually
assumed to be a finite dimensional Hilbert space, i.e., a unitary
space. Using the Hilbert space property of $H$ we can assign, to
every subspace $H_i$, a unique orthogonal projection
$P_i:H\rightarrow H$ onto this subspace. Collections of subspaces
$S=(H;H_1,H_2,\ldots,H_n)$ and
$\tilde{S}=(\tilde{H};\tilde{H}_1,\tilde{H}_2,\ldots,\tilde{H}_n)$
are called unitary equivalent if there exists a unitary operator
$U:H\rightarrow\tilde{H}$ such that $UP_i = \tilde{P_i}U ~ \forall
i=\overline{1,n}$.

A collection of orthogonal projections $\{P_i\}_{i=1}^{n}$ on $H$ is
called irreducible if for any $X\in L(H)$ satisfying $XP_i=P_iX$ for all
$i=\overline{1,n}$ it follows that $X=\lambda I ~ (\lambda \in
\mathbb C)$.

A collection of subspaces $S=(H;H_1,H_2,\ldots,H_n)$ of a Hilbert
space can always be connected with the collection of subspaces
$L=(V=H;V_1=H_1,V_2=H_2,\ldots,V_n=H_n)$ in the linear space $V=H$,
forgetting the scalar product structure. Here unitary equivalent
collections will correspond to isomorphic systems in a linear space
(an isomorphism of systems of subspaces of a Hilbert space is
understood as an isomorphism of the corresponding systems in linear
spaces). The converse, of course, is not true.

For a unitary space, one can also talk about brick
collections and indecomposable collections of subspaces meaning
brick and indecomposable collections of subspaces of $H$ considered
as a linear space. It is easy to see here that indecomposability of
a collection of subspaces implies that the corresponding collection
of the orthogonal projections is irreducible. If a collection of
orthogonal projections is irreducible, then the collection of
subspaces need not be, in general, indecomposable. For example, a
pair of projections onto two nonorthogonal one-dimensional
subspaces in $\mathbb C^2$ will be irreducible but the pair of the
corresponding subspaces is decomposable.

\subsection{}
Irreducible pairs of subspaces exist only in one- and
two-dimensional unitary spaces. A list of the corresponding unitary
nonequivalent pairs of orthogonal projections $\{P_1, P_2\}$ is the
following.
\begin{enumerate}
\item[a)] $\dim H=1$: $\{P_1=0, P_2=0\}$, $\{P_1=1, P_2=0\}$,
  $\{P_1=0, P_2=1\}$, $\{P_1=1, P_2=1\}$;
\item[b)] $\dim H=2$:
  \[
  P_1 = \begin{pmatrix} 1 & 0 \\ 0 & 0 \end{pmatrix}, \qquad
  P_2 =
  \begin{pmatrix} \cos^2\phi & \cos\phi\sin\phi \\ 
    \cos\phi\sin\phi & \sin^2\phi 
  \end{pmatrix}, \quad
  \phi \in (0, \pi/2).
  \]
\end{enumerate}

A description of triples of subspaces of $H$ up to unitary
equivalence is a $*$-wild problem. Even assuming that two of these
spaces are orthogonal, the problem is still
$*$-wild~\cite{KrSam1,KrSam2}.

\subsection{}
For a system of subspaces of a Hilbert space, one can also define
Coxeter functors, $\iiSC{F}$, $\iiSB{F}$, $F^+ =~\iiSC{F}\iiSB{F}$,
and $F^- =~\iiSB{F}\iiSC{F}$, see ~\cite{KRS,Kru02}, which
correspond to the Coxeter functors $\Phi ^+$ and $\Phi ^-$ for
systems in a linear space; if a system $S$ in a Hilbert space $H$
corresponds to a system $L$ in a linear space $V=H$, then $F^+ S$ is
isomorphic to the system $\Phi ^+L$, and $F ^- S$ is isomorphic to
the system $\Phi ^- L$.

The Coxeter functors have the following properties:
$\iiSC{F}\circ\iiSC{F} = \iiSB{F}\circ\iiSB{F} = F^+\circ F^- =
F^-\circ F^+ =\Id$; they preserve the property of the system to be
brick or indecomposable, as well as irreducibility. If
$d=(d_0;d_1,\dots ,d_n)$ is a  dimension of a system $S$, then the
dimension of $\iiSC{F}S$ is $\iiSC{c}d:=(\sum^n_{i=1}
d_i-d_0;d_1,\dots ,d_n)$, and the dimension of $\iiSB{F}S$ is
$\iiSB{c}d:=(d_0;d_0-d_1,\dots ,d_0-d_n)$. Similarly, $F^+S$ has the
dimension $c^+d:=(\iiSC{c}\iiSB{c})d~\{ =\iiSB{c}(\iiSC{c}(d))~\}$,
and the dimension of $F^-S$ is $c^-d:=(\iiSB{c}\iiSC{c})d$.

\subsection{}
One of a natural additional condition on the collection
$S=(H;H_1,H_2,\ldots,H_n)$ so that the problem of unitary
description of the $n$-tuple of subspaces of $H$ becomes solvable is
the orthoscalar condition, which is the linear relation $\sum
\alpha_k P_{H_k} = \alpha_0 I$ for a fixed character $\chi =
(\alpha_0; \alpha_1,\dots ,\alpha_n)$, $\alpha_k>0$~\cite{KRS, Os, Os
  Sam2, KNR}. A remarkable property of such collections of subspaces
is that the Coxeter functors preserve the orthoscalarity
property, although changing the character, in general. Namely, if $S$
is an orthoscalar collection with a character $\chi$, then
$\iiSC{F}S$ is orthoscalar with the character $\iiSB{c}(\chi)$, and
$\iiSB{F}S$ is orthoscalar with the character $\iiSC{c}(\chi)$.

For $n=2$ and $n=3$, orthoscalar collections have a finite Hilbert
type, that is, for any fixed character there exists only a finite
number of unitary nonequivalent irreducible collections of subspaces
satisfying the relation $\sum \alpha_k P_{H_k} = \alpha_0 I$.

For $n=4$ there are two possibilities that depend on the character.
One of them is that there exists a finite number of irreducible
unitary nonequivalent such quadruples (all of them are in a finite
dimensional $H$, although their dimensions could increase when
changing the characters), and they can be obtained from the simplest
ones by applying the Hilbert space version of the Coxeter functors;
formulas for orthogonal projections onto the subspaces of such
quadruples for $\chi=(\gamma; 1,1,1,1)$, $\gamma>0$, can be found
in~\cite{Os Sam1}. Another possibility is a quadruple of
one-dimensional subspaces of a two-dimensional space; formulas for
the orthogonal projections are given in~\cite{KNR}.

\subsection{}
For $n\geqslant5$, a description, up to unitary equivalence, of
$n$-tuples of subspaces,
$$
(H;H_1,H_2,\ldots,H_n),
$$
such that $\sum \alpha_k P_{H_k} = 2I$ is a $*$-wild
problem~\cite{Os Sam1}. It contains the problem of describing triples
  of orthogonal projections $P,Q,R$ such that $Q\bot R$,
$$
P+(I-P)+Q+R+(I-Q-R)=2I.
$$
We remark that for all $\gamma$ such that $ \gamma \in
[\frac{n-\sqrt{n^2-4n}}{n}, \frac{n-\sqrt{n^2+4n}}{n}]$, the problem
of a unitary description of $n$-tuples of subspaces satisfying the
condition $\sum \alpha_k P_{H_k} = \gamma I$ is not a type $I$
problem, see ~\cite{KRS,Shu}.

\subsection{}
One can impose an additional condition on the orthoscalar collection
$$
S=(H;H_1,H_2,\ldots,H_n), \qquad n\geqslant5,
$$
which would allow for a description of unitary nonequivalent
irreducible collections of subspaces. This is the condition that the
subspaces of the collection make a representation of a finite
partially ordered set $\Gamma$ with the number of vertices
$|\Gamma|=n$.

In the case where $\Gamma$ is a primitive partially ordered set,
that is, a partially ordered set consisting of $k$ not connected
linearly ordered sets $p_1^{(j)} < p_2^{(j)} < \dots  < p_{m_i}^{(j)}$,
$\sum_{i=1}^k m_i=n, ~p_i^{(j)} \in \Gamma$, $j=1,\dots ,k$, a study of
their representations in a Hilbert space is the same as studying
collections of subspaces
$(H;\{H_i^{(j)}\}_{i=1,\dots ,m_j}^{j=1,\dots ,k})$ such that $H_i^{(j)}
\subset H_{i+1}^{(j)}$ and
$$
\sum_{j=1}^k \sum_{i=1}^{m_j} a_i^{(j)}P_{H_i^{(j)}}=I
$$
for some fixed collection of positive numbers
$\{a_i^{(j)}\}_{i=1,\dots ,m_j}^{j=1,\dots ,k}$.

Let us now consider the subspaces $\{U_1^{(j)}=H_1^{(j)},~
U_2^{(j)}=H_2^{(j)}\ominus
H_1^{(j)},\dots ,U_{m_j}^{(j)}=H_{m_j}^{(j)}\ominus H_{m_j-1}^{(j)},~
U_{0}^{(j)}=H\ominus H_{m_j}^{(j)}\}$, $j=\overline{1,k}$, that are
mutually orthogonal for fixed $j$. We have that
$$
\sum_{i=0}^{m_j} P_{U_i^{(j)}}=I,\quad j=\overline{1,k},\qquad
\sum_{j=1}^k \sum_{i=1}^{m_j} \beta_i^{(j)}P_{U_i^{(j)}}=I,
$$
where $\beta_i^{(j)} = a_i^{(j)}+\dots +a_{m_i}^{(j)}$,
$i=1,\dots ,m_j$. It is clear that this is the same as to consider a
collection of self-adjoint operators
$A_j=\sum_{i=1}^{m_j}\beta_i^{(j)}P_{U_i^{(j)}}$, that is, such that
the spectrum satisfies $\sigma(A_j) \subset
\{0<\beta_i^{(m_j)}<\dots <\beta_i^{(1)}\}$ and $\sum_{j=1}^k
A_j=I$.  For a study of such operators, see~\cite{KR,VMS,Os Sam2,Alb
  Os Sam} and others.

The representation type of such a problem depends on the tree

\bigskip
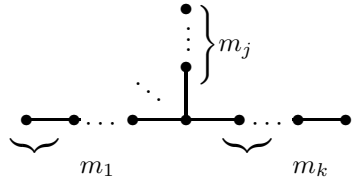
\begin{center}
  \setlength{\unitlength}{.7mm}
\begin{picture}(62,25)(-20,-1)
\thicklines \put(0,0){\circle*{2}} \drawline(0,0)(10,0)
\put(10,0){\circle*{2}} \put(12.3,-2){$\cdots$}
\put(21,0){\circle*{2}} \drawline(20,0)(30,0)
\put(30,0){\circle*{2}}

\drawline(0,0)(-10,0) \put(-10,0){\circle*{2}}
\put(-19,-2){$\cdots$} \put(-21,0){\circle*{2}}
\drawline(-20,0)(-30,0) \put(-30,0){\circle*{2}}

\drawline(0,0)(0,10) \put(0,10){\circle*{2}} \put(-0.3,13){$\vdots$}
\put(0,21){\circle*{2}}

\put(-10,3){$\ddots$}

\put(-33,-3){$\underbrace{~~~~~~~~~~~~~}$}
\put(7,-3){$\underbrace{~~~~~~~~~~~~~}$}

\put(-20,-10){$m_1$} \put(20,-10){$m_k$}

\put(2,13){$\Bigg\}m_j$}

\end{picture}
\end{center}

\bigskip\bigskip

If the tree corresponds to a Dynkin diagram, then there exists only a
finite number of unitary nonequivalent systems of operators
$\{A_j\}_{j=1}^k$ (all of them are finite dimensional) for any fixed
collection of spectrums; if it is a Euclidean graph, which is an
extended Dynkin graph, then depending on the collection of the
spectrums, that are the numbers
$\{\beta_i^{(j)}\}_{i=1,\dots ,m_j}^{j=1,\dots ,k}$, the number of such
unitary nonequivalent irreducible collections is finite or infinite
(all of them are operators on a finite dimensional space, but the
dimension of $H$ could increase when changing the admissible
spectrums). If this tree contains a Euclidean graph as a proper
subgraph, then there always exists a collection
$\{\beta_i^{(j)}\}_{i=1,\dots ,m_j}^{j=1,\dots ,k}$ for which there is an
irreducible collection of infinite dimensional operators
$\{A_j\}_{j=1}^k$.

\subsection{}
Another additional conditions on the collection
$S=(H;H_1,H_2,\ldots,H_n)$ for the problem of unitary description to
become solvable is to choose a configuration of subspaces with given
possible collections $M_{ij} \subset \{(0,0), (1,0), (0,1), (1,1), 0
< \varphi_1^{(ij)} < \varphi_2^{(ij)} < \dots  <
\varphi_{m_{ij}}^{(ij)} < \frac{\pi}{2} \}$ of irreducible
representations for pairs of subspaces $H_i$ and $H_j$, $i\neq j$, $
i,j=1,\dots ,n$. This is a way one obtains various generalizations of
the Temperley-Lieb algebras~\cite{Grah,PSS} and many others.

\section{Unitarization}\label{Sec:4}

\subsection{Definition}
We will say that a collection of subspaces
$L=(V;V_1,V_2,\ldots,V_n)$ of a linear space $V$ can be
unitarized with a character $\chi = (a_0; a_1,\dots ,a_n)$,
$a_0\geqslant 0$, $a_k>0$, $1\leqslant k\leqslant n$, if $H=V$ can
be endowed with a scalar product in such a way that $\sum_{k=1}^n
a_k P_{H_k} = a_0I$, where $P_{H_k}$ are the orthogonal projections
onto $H_k=V_k$. In other words, the collection $L$ is isomorphic to
an orthoscalar collection $S$ with the character $\chi$.

It is clear that a unitarization with a character $\chi = (a_0;
a_1,\dots ,a_n)$ is the same as a unitarization with the character
$\chi^\prime = (a_0/\gamma; a_1/\gamma,\dots ,a_n/\gamma)$, $\gamma>0$.

If $a_0=1$, we will write the character $\chi = (a_0; a_1,\dots ,a_n)$
as $\chi = (a_1,\dots ,a_n)$. Thus, $(a_0;
a_1,\dots ,a_n)=(\frac{a_1}{a_0},\dots ,\frac{a_n}{a_0})$.

It is clear~\cite{KNR} that if an irreducible collection of
subspaces of a Hilbert space is orthoscalar, then it is brick.
Hence, an indecomposable collection in a linear space can be
unitarized if it is brick.

Similarly to~\cite{KNR}, one can conclude that if for a fixed
character $\chi$ there exists a unitarization of an indecomposable
collection $L$, then it is unique, that is, if systems of subspaces
$S$ and $\tilde{S}$ are orthoscalar with the character $\chi$ and
are isomorphic to $L$, then $S$ and $\tilde{S}$ are unitary
equivalent.

In this paper we study the following problems.
\begin{enumerate}
\item[1)] For what brick collections there is a unitarization with
  some character~?
\item[2)] How to describe the characters that allow for a
  unitarization of a given brick collection~?  
\end{enumerate}

Let us remark that statements connected with the unitarization
problem are also contained in~\cite{Ki,CBG} and others.

\subsection{}
It is not difficult to get an answer to the above questions for
collections of $n$ subspaces if $n=2$ or $n=3$.

For $n=2$, we have the following.
\begin{itemize}
\item[] $(\mathbb C; 0, 0)$ can be unitarized with the characters
  $(0; a_1, a_2)$, where $a_1>0,a_2>0$ are arbitrary positive
  numbers.
\item[] $(\mathbb C; \mathbb C, 0)$ can be unitarized with the
  characters $(a_0; a_1, a_2)$, where $a_0=a_1$, $a_1>0, a_2>0$. For
  $(\mathbb C; 0, \mathbb C)$, the answers are obtained by a
  corresponding permutation.
\item[] $(\mathbb C; \mathbb C, \mathbb C)$ can be unitarized with
  the characters $(a_0; a_1, a_2)$, where $a_0>0$, $a_1>0$, $a_2>0$,
  $a_0=a_1+a_2$.
\end{itemize}

Let $n=3$. Note that if one of the subspaces of the collection is
zero, then the problem of describing the characters is reduced to the
problem with fewer subspaces, since the coefficients corresponding
to the zero subspace can be chosen arbitrarily and it does not
influence the others. For example, for $(\mathbb C; \mathbb C,
\mathbb C, 0)$, we get that a unitarization is only possible with
the characters $(a_0; a_1, a_2, a_3)$, where $a_i>0,0\leqslant
i\leqslant3$, $a_0=a_1+a_2$.

For $n=3$ there are only two collections without zero subspaces.
\begin{itemize}
\item[] $(\mathbb C; \mathbb C, \mathbb C, \mathbb C)$ can be
  unitarized with the characters $(a_0; a_1, a_2, a_3)$, where
  $a_i>0,0\leqslant
  i\leqslant3$, $a_0=a_1+a_2+a_3$.
\item[] $(\mathbb C^2; \mathbb C(1,0), \mathbb C(0,1), \mathbb C
  (1,1))$ can be unitarized with the characters $(a_0; a_1, a_2,
  a_3)$, where $0<a_i<a_0,0\leqslant i\leqslant3$,
  $2a_0=a_1+a_2+a_3$.
\end{itemize}

\subsection{}
In the general situation, the following propositions are useful when
studying unitarization of $n$-tuples of subspaces.

\begin{proposition}\label{P:4.1}
  A collection of subspaces $L=(V;V_1,\dots ,V_n)$ can be unitarized
  with a character $\chi$ if and only if $\Phi ^{+}L$ and $\Phi
  ^{-}L$ can be unitarized with the characters $c^{-}\chi$ and $c
  ^{+}\chi$, correspondingly, with the condition that $\Phi
  ^{+}L\neq 0$ and $\Phi ^{-}L\neq 0$, correspondingly.
\end{proposition}

\begin{proof}
  Indeed, let $L$ be isomorphic to an orthoscalar collection $S$
  with a character $\chi$, and let $\Phi ^{+}L\neq 0$. Then $F^+ S
  \neq 0$ and $F^+ S$ is orthoscalar with the character $c^- \chi$.
  Since $F^+ S$ is isomorphic to $\Phi ^{+}L$, we see that $\Phi
  ^{+}L$ can be unitarized with the character $c^- \chi$. Similarly
  we obtain that $\Phi ^{-}L$ can be unitarized with the character
  $c^+ \chi$ if $\Phi ^{-}L\neq 0$. The converse statement follows,
  since $F^+F^-=F^-F^+=\Id$.
\end{proof}

\begin{proposition}\label{P:4.2}
  If a collection of subspaces, $L=(V;V_1,\dots ,V_n)$, of a linear
  space $V$ can be unitarized with some character, then the
  collections $L_0^\prime=(V;V_1,\dots ,V_n,0)$,
  $L_1^\prime=(V;V_1,\dots ,V_n, V)$, and $L^\prime=(V;V_1,\dots,V_n,
  V_k)$, $1\leqslant k \leqslant n$, can be unitarized with some
  characters.
\end{proposition}

\begin{proof}
  Let $L$ be isomorphic to an orthoscalar collection
  $S=(H;H_1,\dots,H_n)$ with a character $\chi =(a_1,\dots ,a_n)$.
  Then the collection $S_0^\prime=(H;H_1,\dots,H_n, 0)$ is
  isomorphic to $L_0^\prime$ and is orthoscalar with the character
  $(a_1,\dots,a_n,1)$, the collection $S_1^\prime=(H;H_1,\dots,H_n,
  H)$ is isomorphic to $L_1^\prime$ and is orthoscalar with the
  character $(a_1/2,\dots,a_n/2,1/2)$, the collection
  $S^\prime=(H;H_1,\dots,H_n, H_k)$ is isomorphic to $L^\prime$ and
  is orthoscalar with the character $(a_1,\dots,a_k/2,
  \dots,a_n,a_k/2)$.
\end{proof}

\begin{theorem}\label{Th:1.1}
  Let $H$ be a linear space of finite dimension $m$. Let
  $S=(H;H_1,H_2,\ldots,H_n)$ be a brick collection of
  one-dimensional subspaces of $H$, that is, $\dim H_i=1$ (note that
  brickness is equivalent to indecomposability in this case). Then
  $S$ is unitarizable with some character.
\end{theorem}

\begin{proof}
  Let us introduce an arbitrary scalar product $(\cdot,\cdot)_1$ on
  $H$. Let $T=\sum P_{H_k}$, where $P_{H_k}$ are orthogonal
  projections onto $H_k$ with respect to the scalar product
  $(\cdot,\cdot)_1$. Since the collection $S$ is brick, the operator
  $T$ is nondegenerate. Also, the operator $T$ is nonnegative, being
  a sum of nonnegative operators. It is clear that $T^{-1}$ is also
  nonnegative.
  
  Define a new scalar product by $(\cdot,\cdot)_2 =
  (T^{-1}(\cdot),\cdot)_1$. Such a definition is correct, since
  $T^{-1}$ is nondegenerate and nonnegative.
  
  Let $v_i \in H_i, i=\overline{1,n}$ $ (v_i\neq0)$. Then for all $ v
  \in H$ and all $i$,
  $$
  P_{H_i}(v) = \frac{{(v,v_i)}_1}{{(v_i,v_i)}_1}\cdot v_i.
  $$
  
  Let $P_{H_i}^{\prime}$ be orthogonal projections onto $H_i$ with
  respect to the scalar product $(\cdot,\cdot)_2$. Then, for all $ v
  \in H$,
  $$
  P_{H_i}^{\prime}(v) = \frac{{(v,v_i)}_2}{{(v_i,v_i)}_2}\cdot v_i
  = \frac{{(T^{-1}(v),v_i)}_1}{{(v_i,v_i)}_2}\cdot v_i=
  \frac{{(v_i,v_i)}_1}{{(v_i,v_i)}_2}\cdot P_{H_i}(T^{-1}(v)).
  $$
  We see that
  $$
  \sum \frac{{(v_i,v_i)}_2}{{(v_i,v_i)}_1}\cdot
  P_{H_i}^{\prime}(v) = \sum P_{H_i}(T^{-1}(v)) = T(T^{-1}(v)) = v,
  $$
  that is
  $$
  \sum \frac{{(v_i,v_i)}_2}{{(v_i,v_i)}_1}\cdot P_{H_i}^{\prime} =
  I,
  $$
  which means that the collection $S$ with the character $\chi =
  \{\frac{{(v_i,v_i)}_2}{{(v_i,v_i)}_1}, i=\overline{1,n}\}$ can be
  unitarized.
\end{proof}

\begin{theorem}\label{Th:1.2}
  \begin{enumerate}
  \item[1)] All collections of subspaces of a linear space, which
    are obtained from brick collections of one-dimensional subspaces
    by adding its copies (see Proposition~\ref{P:4.2}) and by
    applying the Coxeter functors, can be unitarized with some
    character.
  \item[2)] In particular, any brick quadruples of spaces can be
    unitarized with some character.
  \end{enumerate}
\end{theorem}

\begin{proof}
  The first part of the theorem follows directly from
  Propositions~\ref{P:4.1},~\ref{P:4.2}, and Theorem~\ref{Th:1.1}.
  
  To prove the second part, let us recall (see Section~2) that any
  brick quadruple has either discrete or continuous spectrum. In the
  case of a discrete spectrum, the brick quadruples are obtained
  from the simplest ones by applying the Coxeter functors. Since the
  simplest quadruples are one-dimensional, they will be
  unitarizable.  Hence, by Proposition~\ref{P:4.1}, all brick
  quadruples, in the case of a discrete spectrum, are also
  unitarizable. For a continuous spectrum, all brick quadruples have
  the dimension $(2;1,1,1,1)$. Hence, they are unitarizable by
  Theorem~\ref{Th:1.1}.
\end{proof}

\section{A description of characters for which 
  representations of a quadruple of subspaces of $V$ can be
  unitarized}

\subsection{}
All brick quadruples of subspaces are of only two types; the
generalized dimension is a real root (we will call this case
discrete) and an imaginary root (we will call it continuous case).
In the discrete case, for every dimension there is exactly one brick
quadruple. All possible dimensions in this case can be written as
follows:
\begin{itemize}
\item [] $\mathcal{D}_4(2m+1, -1) = (2m+1; m,m,m,m+1), ~ \deff=-1$,
\item[] $\mathcal{D}_4(2m+1, 1) = (2m+1; m+1,m+1,m+1,m), ~ \deff=1$,
\item[] $\mathcal{D}_4(2m, -1) = (2m; m,m,m,m-1), ~ \deff=-1$,
\item[] $\mathcal{D}_4(2m, 1) = (2m; m,m,m,m+1), ~ \deff=1$,
\end{itemize}
and permutations of the spaces $\mathcal{D}_i( \cdot, \cdot),
i=1,2,3$;
\begin{itemize}
\item[] $\mathcal{D}_0(2m+1, -2) = (2m+1; m,m,m,m), ~ \deff=-2$,
\item[] $\mathcal{D}_0(2m, 2) = (2m+1; m+1,m+1,m+1,m+1), ~ \deff=2$.
\end{itemize}

\begin{theorem}\label{Th:2.1}
  Conditions on the character such that every brick quadruple can be
  unitarized in the discrete case can be written as follows:
  \begin{itemize}
  \item[] $\mathcal{D}_4(2m+1,-1):~m\cdot
    def(a)+a_i>0,~i=\overline{1,4},~m\cdot def(a)=a_4-a_0$, where
    $def(a) = 2a_0-a_1-a_2-a_3-a_4$,
  \item[] $\mathcal{D}_4(2m+1,1):~a_i-m\cdot
    def(a)>0,~i=\overline{1,4},~(m+1)\cdot def(a)+a_4-a_0=0$,
  \item[] $\mathcal{D}_4(2m,-1):~(m-1)\cdot
    def(a)+a_0-a_i>0,~i=\overline{1,4},~m\cdot def(a)+a_4=0$,
  \item[] $\mathcal{D}_4(2m,1):~a_0-a_i-m\cdot
    def(a)>0,~i=\overline{1,4},~m\cdot def(a)=a_4$,
  \item[] $\mathcal{D}_0(4m+1,-2):~m\cdot
    def(a)+a_i>0,~i=\overline{1,4},~2m\cdot def(a)+a_0=0$,
  \item[] $\mathcal{D}_0(4m+1,2):~a_i-m\cdot
    def(a)>0,~i=\overline{1,4},~a_0-(2m+1)\cdot def(a)=0$,
  \item[] $\mathcal{D}_0(4m+3,-2):~m\cdot
    def(a)+a_0-a_i>0,~i=\overline{1,4},~(2m+1)\cdot def(a)+a_0=0$,
  \item[] $\mathcal{D}_0(4m+3,2):~a_0-a_i-m\cdot
    def(a)>0,~i=\overline{1,4},~a_0-(2m+2)\cdot def(a)=0$.
  \end{itemize}
\end{theorem}  

\begin{proof}
  Using $\iiSC{c}$ and $\iiSB{c}$ we can write that
  $$
  \begin{array}{c}
    \mathcal{D}_4(2m+1,-1) =
    {(\iiSB{c}\iiSC{c})}^{2m}\mathcal{D}_4(1,-1),\\[2mm]
    \mathcal{D}_4(2m+1,1) =
    {(\iiSB{c}\iiSC{c})}^{2m}\iiSB{c}\mathcal{D}_4(1,-1),\\[2mm]
    \mathcal{D}_4(2m,-1) =
    {(\iiSB{c}\iiSC{c})}^{2m-1}\mathcal{D}_4(1,-1),\\[2mm]
    \mathcal{D}_4(2m,1) =
    {(\iiSB{c}\iiSC{c})}^{2m-1}\iiSB{c}\mathcal{D}_4(1,-1),\\[2mm]
    \mathcal{D}_0(2m+1,-2) =
    {(\iiSB{c}\iiSC{c})}^{m}\mathcal{D}_0(1,-2),\\[2mm]
    \mathcal{D}_0(2m+1,2) = 
    {(\iiSB{c}\iiSC{c})}^{m}\iiSB{c}\mathcal{D}_0(1,-2).
  \end{array}
  $$
  Properties of Coxeter functors show that if there exists a
  brick collection of subspaces of dimension $d$, which can be
  unitarized with a character $\chi = (a_0; a_1,a_2,a_3,a_4)$, then
  there exists a brick collection of subspaces with the dimension
  $\iiSB{c}(d)$, which can be unitarized with the character
  $\iiSC{c}(\chi)$, as well as a brick collection of subspaces with
  the dimension $\iiSC{c}(d)$, unitarizable with the character
  $\iiSB{c}(\chi)$. Thus, knowing the characters that permit the
  simplest quadruples to be unitarized, we can find the characters
  allowing a unitarization of other quadruples in the discrete case.
  
  It is clear that a quadruple with the dimension
  $\mathcal{D}_4(1,-1)=(1;0,0,0,1)$ can be unitarized with $\chi_4 =
  (a_0; a_1,a_2,a_3,a_4)$ if and only if $a_0=a_4, ~ a_i>0,
  i=\overline{0,4}$; a quadruple with the dimension
  $\mathcal{D}_0(1,-2)=(1;0,0,0,0)$ can be unitarized with $\chi_1 =
  (a_0; a_1,a_2,a_3,a_4)$ if and only if $a_0=0, ~ a_i>0,
  i=\overline{1,4}$.
  
  We get, for example, that for the dimension
  $\mathcal{D}_4(2m+1,-1) =
  {(\iiSB{c}\iiSC{c})}^{2m}\mathcal{D}_4(1,-1)$, a collection of
  subspaces can be unitarized with a character $\chi$ if and only if
  $\chi = {(\iiSC{c}\iiSB{c})}^{2m}\chi_4$. In other words, if
  $\chi$ is a character that allows for a unitarization of a
  quadruple with the dimension
  ${(\iiSB{c}\iiSC{c})}^{2m}\mathcal{D}_4(1,-1)$, then
  ${(\iiSB{c}\iiSC{c})}^{2m}\chi$ and $\chi_4$ must satisfy the same
  conditions.

  It is not difficult to calculate that
  \begin{equation*}
    \begin{aligned}
      {(\iiSB{c}\iiSC{c})}^{2m}(a_0; a_1,a_2,a_3,a_4) &= (2m\cdot
      \deff(a)+a_0; 
      m\cdot \deff(a)+a_1,m\cdot \deff(a)+a_2,\\
      &m\cdot
      \deff(a)+a_3,m\cdot \deff(a)+a_4),\\
      {(\iiSB{c}\iiSC{c})}^{2m}\iiSB{c}(a_0; a_1,a_2,a_3,a_4) & =
      (2m\cdot \deff(a)+a_0;
      m\cdot \deff(a)+a_0-a_1,m\cdot \deff(a)+a_0-a_2,\\
      &m\cdot
      \deff(a)+a_0-a_3,m\cdot
      \deff(a)+a_0-a_4),\\
      {(\iiSB{c}\iiSC{c})}^{2m+1}(a_0; a_1,a_2,a_3,a_4)& =
      ((2m+1)\cdot \deff(a)+a_0;
      m\cdot \deff(a)+a_0-a_1,\\
      &m\cdot \deff(a)+a_0-a_2,
      m\cdot
      \deff(a)+a_0-a_3,m\cdot \deff(a)+a_0-a_4),\\
      {(\iiSB{c}\iiSC{c})}^{2m+1}\iiSB{c}(a_0; a_1,a_2,a_3,a_4) &=
      ((2m+1)\cdot \deff(a)+a_0; 
      (m+1)\cdot \deff(a)+a_1,\\
      &(m+1)\cdot
      \deff(a)+a_2,
      (m+1)\cdot
      \deff(a)+a_3,(m+1)\cdot \deff(a)+a_4).\\
      {\iiSC{c}(\iiSB{c}\iiSC{c})}^{2m}(a_0; a_1,a_2,a_3,a_4) &=
      (a_0-(2m+1)\cdot def(a); a_1-m\cdot def(a),\\
      &a_2-m\cdot
      def(a),a_3-m\cdot def(a),a_4-m\cdot def(a))\\
      {\iiSC{c}(\iiSB{c}\iiSC{c})}^{2m+1}(a_0; a_1,a_2,a_3,a_4) &=
      (a_0-(2m+2)\cdot def(a); a_0-a_1-(m+1)\cdot def(a),\\
      &a_0-a_2-(m+1)\cdot def(a),
      a_0-a_3-(m+1)\cdot def(a),\\
      &a_0-a_4-(m+1)\cdot def(a)).
    \end{aligned}
  \end{equation*}
  By using these formulas and the conditions on the character that
  allow for a unitarization of the simplest quadruples
  $\mathcal{D}_4(1,-1)$ and $\mathcal{D}_0(1,-2)$, we obtain
  conditions on the character for other collections.
\end{proof}

Let us remark that similar considerations were used in~\cite{KPS},
although for a different purpose.

Also note that Theorem~\ref{Th:2.1} shows that any brick quadruple
with the character $(\gamma; 1,1,1,1)$ can be unitarized in the
discrete case, which was proved in~\cite{MS2}. Namely, for a
quadruple of dimension $d=(d_0; d_1,d_2,d_3,d_4)$, one should take
$\gamma=2-\deff(d)/d_0$. A simple check shows that the conditions of
Theorem~\ref{Th:2.1} are satisfies.

\subsection{}
Let us consider the continuous case. Here, brick collections exist
only for the dimension $(2;1,1,1,1)$. There is a series of
quadruples parametrized with $\mu \in \mathbb C \setminus \{0,1\}$,
$$
S_{\mu} = (<e_1, e_2>; <e_1>, <e_2>, <e_1+\mu e_2>, <e_1+e_2>),
$$
and two degenerate quadruples,
$$
S_{3,4} = (<e_1, e_2>; <e_1>, <e_2>, <e_1+e_2>, <e_1+e_2>),
$$
and $S_{i,j}$ obtained by permutation of the subspaces.

\begin{theorem}\label{Th:2.2}
  \begin{itemize}
  \item [a)] The degenerate representation of $S_{3,4}$ can be
    unitarized with the character $\chi = (a_0; a_1,a_2,a_3,a_4)$ if
    $a_1+a_2>a_3+a_4,~ a_1<a_2+a_3+a_4,~a_2<a_1+a_3+a_4,~2a_0 =
    a_1+a_2+a_3+a_4,~ a_i>0$.
    
  \item[b)] All nondegenerate representations can be unitarized with
    $\chi = (a_0; a_1,a_2,a_3,a_4)$ if and only if 
    $$
    2a_i<\sum_{j=1}^{4}a_j,\quad i=\overline{1,4},\qquad
    2a_0=\sum_{j=1}^{4}a_j,\quad 0<a_i<a_0.
    $$
  \end{itemize}
\end{theorem}

\begin{proof}
  \textit{a)} Clearly, the degenerate representation of $S_{3,4}$
  can be unitarized with the character $\chi = (a_0;
  a_1,a_2,a_3,a_4)$ if and only if the representation of the triple
  of subspaces, $(<e_1, e_2>; <e_1>, <e_2>, <e_1+e_2>)$, can be
  unitarized with the character $(a_0; a_1,a_2,a_3+a_4)$, that is,
  if $a_1+a_2>a_3+a_4,~ a_1<a_2+a_3+a_4,~a_2<a_1+a_3+a_4,~2a_0 =
  a_1+a_2+a_3+a_4,~ a_i>0$.
  
  \textit{b)} It is clear that the conditions imposed on the
  character are necessary. Let us prove that they are sufficient.
  
  The proof is similar to considerations in~\cite{MOY}. Fix
  $a_1\leqslant a_2\leqslant a_3\leqslant a_4$, $a_1+a_2+a_3+a_4=2$,
  and use the formulas obtained in [KNR] for solving the
  equation $a_1P_1+a_2P_2+a_3P_3+a_4P_4=I$ in the dimension
  $(2;1,1,1,1)$, where $P_i$ are orthogonal projections,
  \begin{eqnarray*}
    P_1 &=& \frac{1}{2a_1\lambda}
    \begin{pmatrix}
      (\lambda-A)(\lambda+B)
      & \sqrt{-(\lambda^2-A^2)(\lambda^2-B^2)} \\ 
      \sqrt{-(\lambda^2-A^2)(\lambda^2-B^2)} & 
      -(\lambda+A)(\lambda-B)
    \end{pmatrix},\\
    P_2 &=& \frac{1}{2a_2\lambda}
    \begin{pmatrix} -(\lambda-D)(\lambda+C) &
      e^{ix}\sqrt{-(\lambda^2-D^2)(\lambda^2-C^2)}
      \\ 
      e^{-ix}\sqrt{-(\lambda^2-D^2)(\lambda^2-C^2)} 
      & (\lambda+D)(\lambda-C)
    \end{pmatrix},
    \\
    P_3 &=& \frac{1}{2a_3\lambda}
    \begin{pmatrix} -(\lambda-D)(\lambda-C) &
      -e^{ix}\sqrt{-(\lambda^2-D^2)(\lambda^2-C^2)}
      \\ 
      -e^{-ix}\sqrt{-(\lambda^2-D^2)(\lambda^2-C^2)} 
      & (\lambda+D)(\lambda+C)
    \end{pmatrix},
    \\
    P_4 &=& \frac{1}{2a_4\lambda}
    \begin{pmatrix} (\lambda+A)(\lambda+B) &
      -\sqrt{-(\lambda^2-A^2)(\lambda^2-B^2)}
      \\ 
      -\sqrt{-(\lambda^2-A^2)(\lambda^2-B^2)} 
      & -(\lambda-A)(\lambda-B)
    \end{pmatrix},
  \end{eqnarray*}
  $A\leqslant \lambda \leqslant \min(B,D), ~ 0\leqslant x \leq
  2\pi$, where $A=(a_4-a_1)/2, B=(a_4+a_1)/2, C=(a_3-a_2)/2,
  D=(a_3+a_2)/2$.

  If $A=0$, then $a_1=a_2=a_3=a_4=\frac12$. This case was considered
  in~\cite{MS2}, where, in particular, it was shown that any brick
  quadruple can be unitarized in the continuous case with the
  character $(2;1,1,1,1)$. So, we assume that $A>0$.
  
  Let us show that $(\Imp P_1, \Imp P_2, \Imp P_3, \Imp P_4)$ give
  all brick nondegenerate quadruples of the dimension $(2;1,1,1,1)$
  when $\lambda$ and $x$ are changing.
  
  Denote 
  $$
  K_1=\sqrt{\frac{(\lambda+A)(B-\lambda)}
    {(\lambda-A)(B+\lambda)}},\quad
  K_2=\sqrt{\frac{(\lambda-C)(D+\lambda)}
    {(\lambda+C)(D-\lambda)}},\quad
  K_3=\frac{\lambda+C}{\lambda-C}K_2,\quad
  K_4=\frac{\lambda-A}{\lambda+A}K_1.
  $$
  Then we have
  $$
  \begin{array}{ll}
    \Imp P_1 = <e_1+K_1e_2>,& 
    \Imp P_2 = <e_1+e^{-ix}K_2e_2>,\\[2mm]
    \Imp P_3 = <e_1-e^{-ix}K_3e_2>, & \Imp P_4 = <e_1-K_4e_2>.
  \end{array}
  $$
  If $P_i\neq P_j, i\neq j$, this system will be isomorphic to
  the system 
  $$
  (<e_1, e_2>; <e_1>, <e_2>, <e_1+\mu e_2>, <e_1+e_2>),
  $$
  where 
  $$
  \mu = \frac{1}{(K_1+K_4)(K_2+K_3)}
  (K_1K_2+K_3K_4+K_1K_4e^{ix}+K_2K_3e^{-ix}).
  $$
  Substituting, we get
  \begin{eqnarray*}
    \mu &=& \frac{1}{2} - \frac{AC}{2\lambda^2} +
    \frac{1}{4\lambda^2}K_1K_2^{-1}(\lambda-A)(\lambda-C)e^{ix} +
    \frac{1}{4\lambda^2}K_1^{-1}K_2(\lambda+A) 
    (\lambda+C)e^{-ix}\\[2mm]
    &=&\frac{1}{2} - \frac{AC}{2\lambda^2} +
    \frac{1}{4\lambda^2}\sqrt{(\lambda^2-A^2)(\lambda^2-C^2)}
    \begin{pmatrix}
      \sqrt{\frac{(B-\lambda)
          (D-\lambda)}{(B+\lambda)(D+\lambda)}}e^{-ix} +
      \sqrt{\frac{(B+\lambda)(D+\lambda)}
        {(B-\lambda)(D-\lambda)}}e^{ix}
    \end{pmatrix}.
  \end{eqnarray*}
  
  For a fixed $\lambda$, we have that $\mu$, considered as a
  function of $x$, is an ellipse in $\mathbb C$. If $\lambda$
  increases to $\min(B,D)$, then this ellipse extends to infinity.
  If $\lambda$ approaches $A$, then the ellipse contracts to the
  point $(\frac12-\frac{C}{2A})$. Thus $\mu$ can be made arbitrary
  distinct from $0$ and $1$ when varying $\lambda$ and $x$.
\end{proof}

\begin{remark}
  Let us remark that the conditions on the character in part b),
  Theorem~\ref{Th:2.2}, do not depend on the parameter $\mu$, that
  is, all $S_\mu$, $\mu \in \mathbb{C}\setminus \{0,1\}$, can be
  unitarized with the same characters.
\end{remark}

\end{document}